\documentclass[12pt]{article}

\usepackage{amsmath,amsfonts,amssymb,amsthm,amsrefs}
\usepackage{enumitem}

\title{Leaf Management}
\author{Jeffry L.~Hirst}

\date {December 23, 2018}

\setlist[enumerate]{label=\rm{(\arabic*)}, ref=\arabic*}

\theoremstyle{plain}
\newtheorem{thm}{Theorem}
\newtheorem{coro}[thm]{Corollary}  
\newtheorem{lemma}[thm]{Lemma}
\newtheorem{prop}[thm]{Proposition}

\newcommand{\nat}{\mathbb N}
\newcommand{\rca}{{\sf{RCA}}_0}
\newcommand{\aca}{{\sf{ACA}}_0}
\newcommand{\atr}{{\sf{ATR}}_0}
\newcommand{\poo}{\Pi^1_1 {\rm -} {\sf{CA}}_0}
\newcommand{\wf}{{\sf{WF}}}
\newcommand{\wfh}{\widehat{\sf{WF}}}

\newcommand{\ltw}{<_{\rm W}}
\newcommand{\lews}{\leq_{\rm sW}}
\newcommand{\ltws}{<_{\rm sW}}

\newcommand{\eqws}{\equiv_{\rm sW}}
\newcommand{\lpo}{{\sf{LPO}}}
\newcommand{\ssep}{{\Sigma^1_1 {\rm -} {\sf{SEP}}}}
\newcommand{\sca}{{\Sigma^1_1}{\rm -}{\sf{CA}}^-}
\newcommand{\ptt}{{{\sf{PTT}}_1}}
\newcommand{\pk}{{\sf{PK}}}
\newcommand{\hpc}{{\sf{HPC}}}
\newcommand{\hpch}{\widehat{\sf{\hpc}}}
\newcommand{\leaf}[1]{{\sf{leaf}}(#1)}
\newcommand{\wleaf}{\sf{leaf}}
\newcommand{\cat}{^\frown}
\newcommand{\wbl}{{\sf{bleaf}}}
\newcommand{\lpoh}{\widehat\lpo}
\newcommand{\lh}[1]{{|#1{}|}}
\newcommand{\mdot}{{\frac{\cdot}{~~}}}

\begin{document}

\maketitle

\begin{abstract}
Finding the set of leaves for an unbounded tree is a nontrivial process
in both the Weihrauch and reverse mathematics settings.
Despite this, many combinatorial principles for trees are equivalent
to their restrictions to trees with leaf sets.  For example, let $\wfh$ denote
the problem of choosing which trees in a sequence are well-founded, and let
$\pk$ denote the problem of finding the perfect kernel of a tree.  Let
$\wfh_L$ and $\pk_L$ denote the restrictions of these principles to trees with
leaf sets.  Then $\wfh$, $\wfh_L$, $\pk$, and $\pk_L$ are all equivalent to $\poo$
over $\rca$, and all strongly Weihrauch equivalent.
\end{abstract}

\subsubsection*{Introduction}
The first section of this paper shows that for unbounded trees, finding leaf sets
is a nontrivial process.  The second section describes an algorithm for transforming trees
into trees with leaf sets in such a way that properties related to infinite paths and perfect subtrees
are preserved.  The main equivalence results are presented in this section.  The paper closes
with a section containing an application to hypergraphs, where the use of a combinatorial
principle restricted to sequences of trees with leaf sets is central to the proof of a Weihrauch
equivalence.

All relevant background information on reverse mathematics can be found in Simpson's text \cite{simpson}.
For background on Weihrauch analysis, see the work of
Brattka, Gherardi, and Pauly \cite{bgp}.

\subsubsection*{Leaf sets}

In second order arithmetic settings, a tree is encoded by a set of finite sequences of natural numbers that is closed under initial subsequences.
For any finite sequence $\sigma$, let $\lh{\sigma}$ denote the length of $\sigma$.
A leaf in a tree is a sequence that has no extensions in the tree.  For a tree $T$, let $\leaf{T}$ denote the set of leaves of $T$.
A function $b: \nat \to \nat$ is a bounding function for $T$ if for every $\sigma \in T$ and for every $i < \lh{\sigma}$, $\sigma (i) \le b(i)$.
If a tree $T$ has a bounding function, little set comprehension is required to calculate $\leaf{T}$.

\begin{prop}
$(\rca )$  If $b$ is a bounding function for the tree $T$, then $\leaf{T}$ exists.
\end{prop}

\begin{proof}
Working in $\rca$, the sequence $\sigma$ is a leaf of $T$ if and only if $\sigma \in T$ and
for each $j \le b (\lh{\sigma}+1)$ we have $\sigma \cat j \notin T$.  Thus, the set $\leaf{T}$
is computable using $T$ and $b$ as parameters, and exists by recursive comprehension.
\end{proof}

Let $\wbl$ denote the Weihrauch problem that accepts a tree $T$ and a bounding function $b$ as inputs
and outputs the set $\leaf{T}$.  For a bounded tree, the preceding proof describes a process for computing
the leaf set.  Consequently, $\wbl$ is at the lowest level of the strong Weihrauch hierarchy, as stated in the following
proposition.

\begin{prop}
$\wbl \eqws 0$.
\end{prop}

Finding leaf sets for trees without bounding functions is nontrivial.  The formulation
of $\lpo$ parallelized in the next proposition is the one preceding Theorem 6.7
of Brattka, Gherardi, and Pauly \cite{bgp}, and the Boolean negation of Definition 2.6 of
Brattka and Gherardi \cite{bg}.

\begin{prop}\label{P2}
$(\rca)$  The following are equivalent.
\begin{enumerate}
\item $\aca$.\label{P2a}
\item For every tree $T$, the set $\leaf{T}$ exists.\label{P2b}
\item $\lpoh$:
If $\langle p_i \rangle _{i \in \nat}$ is a sequence
of sequences of natural numbers, then there is a function
$z : \nat \to \{0,1\}$ such that for each $i$, $z(i)=1$ if and only if
$\exists n (p_i (n) = 0 )$.\label{P2c}
\end{enumerate}
\end{prop}

\begin{proof}
We will work in $\rca$ throughout the proof.  To see that (\ref{P2a}) implies
(\ref{P2b}), note that $\sigma \in \leaf{T}$ if and only if
$\sigma \in T$ and $\forall j (\sigma \cat j \notin T)$.  Thus arithmetical comprehension
proves the existence of $\leaf {T}$.

To see that (\ref{P2b}) implies (\ref{P2c}), assume (\ref{P2b})
and let $\langle p_i \rangle_{i \in \nat}$ be an instance of $\lpoh$.
Consider the tree $T$ constructed from $\langle p_i \rangle_{i \in \nat}$
as follows.  Every finite sequence of ones is in $T$.  For each $n \in \nat$
the sequence $\sigma_n$ consisting of $n+1$ ones followed by a zero is in $T$.
The sequence $\sigma_n\cat j$ is in $T$ if and only if $p_n (j)=0$ and
$\forall i< j (p_n(i)\neq 0)$.  The set of sequences $T$ exists
by recursive comprehension and is a tree because it is closed
under initial segments.  Apply (\ref{P2b}) to find $\leaf{T}$.  
The function $z: \nat \to \{0,1\}$ defined by
$z(n) = 1$ if and only if $\sigma_n \notin \leaf{T}$ exists by recursive comprehension
and is a solution of the instance of $\lpoh$.

To complete the proof, by Lemma III.1.3 of Simpson \cite{simpson}, it suffices to use
(\ref{P2c}) to find the range of an injection.  Let $f: \nat \to \nat$ be an injection.  By
recursive comprehension, we can find the sequence $\langle p_i \rangle_{i \in \nat}$ defined
by $p_i(n)=0$ if $f(n)=i$, and $p_i(n)=1$ otherwise.
Apply (\ref{P2c}) to find $z$ such that $z(i)=1$ if and only if $\exists n (p_i (n)=0)$.
Then the set $\{ i \mid z(i) = 1\}$ is the range of $f$, is computable from $z$, and
so exists by recursive comprehension.
\end{proof}

By virtue of our circuitous reverse mathematics treatment, we can easily
prove the related Weihrauch reducibility result.  Let $\wleaf$ denote
the problem that accepts a tree $T$ as input and outputs the leaf set $\leaf{T}$.

\begin{prop}\label{P3}
$\wleaf \eqws \lpoh$.
\end{prop}

\begin{proof}
The proof that (\ref{P2b}) implies (\ref{P2c}) for Proposition~\ref{P2} also shows that
$\lpoh \lews \wleaf$.  To prove the reverse relation, fix $T$ and let $\langle \sigma_i \rangle_{i \in \nat}$
be an enumeration of the sequences in $T$.  For each $i$, let $p_i(j) =0$ if $\sigma_i \cat j \in T$ and
let $p_i (j) = 1$ otherwise.  If $z$ is a solution to this instance of $\lpoh$, then it is also a characteristic
function for $\leaf{T}$.
\end{proof}

\subsubsection*{Transforming trees}

As shown in the previous section,  finding the leaf set for an arbitrary tree is a nontrivial process
in both the reverse mathematics and Weihrauch settings.  However, in many cases it is possible
to uniformly transform trees into trees with leaf sets while preserving many Weihrauch
equivalences and equivalence theorems of reverse mathematics.
The transformation can be defined using the following operations on finite sequences.  For
every $\sigma \in \nat^{<\nat}$, let $\sigma + 1$ denote the sequence with exactly the same length as $\sigma$ such
that for all $n < \lh{\sigma}$, $(\sigma  + 1)(n) = \sigma(n) + 1$.
For example, $\langle 1, 3, 5 \rangle + 1 = \langle 2,4,6\rangle$.
Similarly, define $\sigma \mdot 1$ so that $(\sigma \mdot 1 ) (n) = \sigma(n) \mdot 1$.
Our main tree transformation is $T^*$ as defined in the following theorem.
Na\"\i{}vely, $T^*$ is created by adding $1$ to every node of $T$ and attaching a leaf
labeled $0$ to each positive node.

\begin{lemma}\label{B1}
$(\rca)$
Suppose $T \subset \nat^{<\nat}$ is a tree.  The following are also trees:
$T^- =\{ \sigma \mdot 1 \mid \sigma \in T\}$,
$T^+ =\{ \sigma +1 \mid \sigma \in T\}$, and
$T^* =T^+ \cup \{ \sigma \cat 0  \mid \sigma \in T^+\}$.
Furthermore, given a sequence $\langle T_i \rangle_{i \in \nat}$, we can find
the sequence $\langle T_i^* , \leaf{T_i^* } \rangle$.
\end{lemma}

\begin{proof}
Working in $\rca$, it is easy to use recursive comprehension to prove the existence of the sets
$T^-$, $T^+$, and $T^*$.  The initial segments of the shifts $\sigma \mdot 1$ and $\sigma + 1$ are
the shifts of initial segments of $\sigma$, so $T^-$ and $T^+$ are trees.  A proper initial segment
of a sequence in the set $\{ \sigma \cat 0  \mid \sigma \in T^+\}$ is an initial segment of an element
of $T^+$, so $T^*$ is also a tree.

To complete the proof, suppose $\langle T_i \rangle_{i \in \nat}$ is a sequence of trees.  For each
$i$, $\sigma \in T_i^*$ if and only if the last element of $\sigma$ is positive and $\sigma \mdot 1 \in T_i$,
or if $\sigma = \tau \cat 0$ and $\tau \mdot 1 \in T_i$.
Thus recursive comprehension implies that $\langle T_i^* \rangle_{i \in \nat}$ exists.  For each $i$,
the sequence $\sigma \in \leaf {T_i^* }$ if and only if $\sigma \in T_i^*$ and the last entry of $\sigma$ is $0$.
Thus $\rca$ can prove that the sequence of pairs $\langle T_i^* , \leaf{T_i^* } \rangle$ exists.
\end{proof}

The trees $T$ and $T^*$ share many properties.  Information about paths and subtrees of one can be
uniformly transformed to information about the other.  As described in Simpson \cite{simpson}*{Definition I.6.6},
a subtree $S$ if $T$ is perfect if every sequence in $T$ has incompatible extensions in $T$.  The perfect
kernel of $T$ is the union of all the perfect subtrees of $T$.

\begin{thm}\label{B2}
$( \rca )$  A tree $T$ and the transform $T^*$ satisfy the following.
\begin{enumerate}
\item $T$ is well-founded if and only if $T^*$ is well-founded.\label{B2a}
\item  $T$ has at most one path if and only if $T^*$ has at most one path.\label{B2b}
\item  $S$ is a perfect subtree of $T$ if and only if $S^+$ is a perfect subtree of $T^*$.\label{B2c}
\item  $K$ is the perfect kernel of $T$ if and only if $K^+$ is the perfect kernel of $T^*$.\label{B2d}
\end{enumerate}
\end{thm}

\begin{proof}
Each part follows from the fact that the map taking $\sigma$ to $\sigma+1$ is a bijection between
the paths of $T$ and those of $T^*$ and also between the perfect subtrees of $T$ and those of $T^*$.
\end{proof}

The next two theorems list familiar equivalences for tree statements that continue to hold when restricted
to trees with leaf sets.  For both proofs, the central tool is the transformation from $T$ to $T^*$.  In the following
theorem, the labels
used for the combinatorial principles are consistent with those used for the associated Weihrauch problems
by Kihara, Marcone, and Pauly \cite{kmp}.

\begin{thm}\label{B3}
$( \rca )$  The following are equivalent.
\begin{enumerate}
\item  $\atr$.\label{B3a}
\item  {\rm The $\Sigma^1_1$ separation principle:}  For any $\Sigma^1_1$ formulas $\varphi _0 (n)$
and $\varphi _ 1 (n)$ containing no free occurrences of $Z$, if 
$\neg \exists n ( \varphi_0 (n) \land \varphi_1 (n) )$, then
\[\exists Z \forall n ((\varphi_0 (n) \to n \in Z ) \land (\varphi_1 (n) \to n \notin Z)).\]\label{B3b}
\item  $\ssep :$  If $\langle T_{0,i} \rangle_{i \in \nat}$ and $\langle T_{1,i} \rangle_{i \in \nat}$
are sequences of trees such that for each $i$, at most one of $T_{0,i}$ and $T_{1,i}$ has an
infinite path, then there is a set $Z$ such that for all $n$, $T_{0,n}$ has an infinite path implies $n \in Z$
and $T_{1,n}$ has an infinite path implies $n \notin Z$.\label{B3c}
\item $\ssep_L :$  Item (\ref{B3c}) for 
$\langle T_{0,i} , \leaf {T_{0,i}}  \rangle_{i \in \nat}$ and $\langle T_{1,i} , \leaf {T_{1,i}}  \rangle_{i \in \nat}$,
sequences of trees with leaf sets.
\label{B3d}
\item \label{B3e}
$\sca :$  If $\langle T_i \rangle_{i \in \nat}$ is a sequence of trees each with at most one infinite path,
then there is a set $Z$ such that for all $n$, $n \in Z$ if and only if $T_n$ has an infinite path.
\item \label{B3f}
$\sca_L :$  Item \ref{B3e} for sequences of trees with leaf sets.
\item  \label{B3g}
$\ptt :$  If $T$ has uncountably many paths then $T$ has a non-empty perfect subtree.
\item  \label{B3h}
$\ptt _L :$  Item \ref{B3g} for trees with leaf sets.
\end{enumerate}
\end{thm}

\begin{proof}
The equivalence of (\ref{B3a}) and (\ref{B3b}) is Theorem V.5.1 of Simpson \cite{simpson}.
The existence of an infinite path in a tree can be written as a $\Sigma^1_1$ formula,
so (\ref{B3b}) implies (\ref{B3c}).  To prove the converse, use a bootstrapping argument,
proving $\aca$ from (\ref{B3c}) by creating a sequence of pairs of linear trees that compute
the range of an injection.  Then use $\aca$ and (\ref{B3c}) to derive (\ref{B3b}) by an
application of Lemma 3.14 of Friedman and Hirst \cite{fh}.  The equivalence of (\ref{B3e}) and (\ref{B3a})
is Theorem V.5.2 of Simpson \cite{simpson}, and the equivalence of (\ref{B3g}) and (\ref{B3a}) is
Theorem V.5.5 of Simpson \cite{simpson}.  Item (\ref{B3d}) is a restriction of (\ref{B3c}), so (\ref{B3c})
implies (\ref{B3d}) trivially.  The converse is an immediate consequence of Theorem \ref{B2}.
Similarly, (\ref{B3e}) and (\ref{B3f}) are equivalent, as are (\ref{B3g}) and (\ref{B3h}).
\end{proof}

To avoid confusion with the subsystem $\poo$,
in following theorem we use $\wfh$ as a label
for the combinatorial principle denoted by $\Pi^1_1{\rm -}{{\sf {CA}}}$ 
in the article of Kihara, Marcone, and Pauly \cite{kmp}.  Note that $\wfh$ is the
infinite parallelization of the the principle $\wf$, that takes a tree as an input and
outputs a $1$ if the tree is well-founded and a $0$ otherwise.

\begin{thm}\label{B4}
$(\rca )$  The following are equivalent:
\begin{enumerate}
\item\label{B4a}  $\poo :$  If $\varphi(n)$ is a $\Pi^1_1$ formula, then there is a set $Z$
such that for all $n$, $n \in Z$ if and only if $\varphi (n)$.
\item\label{B4b}  $\wfh :$  If $\langle T_i \rangle_{i \in \nat}$ is a sequence of trees, then there is a set $Z$ such that
for all $n$, $n \in Z$ if and only if $T$ has no infinite path.
\item\label{B4c}  $\wfh _L :$  Item (\ref{B4b}) for sequences of trees with leaf sets.
\item\label{B4d}  $\pk :$ Every tree has a perfect kernel.
\item\label{B4e}  $\pk _L :$  Item (\ref{B4d}) for trees with leaf sets.
\end{enumerate}
\end{thm}

\begin{proof}
The equivalence of (\ref{B4a}) and (\ref{B4b}) is Theorem VI.1.1 of Simpson \cite{simpson}.
The equivalence of (\ref{B4a}) and (\ref{B4d}) is Theorem VI.1.3 of Simpson \cite{simpson}.
The restriction (\ref{B4c}) follows trivially from (\ref{B4b}), and the converse follows immediately
from Theorem \ref{B2}.  By a similar argument, (\ref{B4d}) and (\ref{B4e}) are equivalent.
\end{proof}

We now turn to the Weihrauch analogs of the preceding results.  The main tool is the computability
theoretic version of Lemma \ref{B1}.

\begin{lemma}\label{B5}
There is a uniformly computable map from trees $T$ to $T^-$, and invertible uniformly
computable maps from $T$ to $T^+$ and $T^*$.
Also, there is a computable functional mapping sequences
of trees $\langle T_i \rangle_{i \in \nat}$ to
$\langle T_i^* , \leaf{T_i^* } \rangle$.
\end{lemma}

\begin{proof}
The processes described at the beginning of the section are uniformly computable, and
for $T^+$ and $T^*$, uniformly computably invertible.  Leaf sets are uniformly computable
for trees of the form $T^*$.
\end{proof}

The following Weihrauch analog of Theorem \ref{B3} is based on the results of Kihara, Marcone, and Pauly \cite{kmp}.

\begin{thm}\label{B6}
$\ptt \eqws \ptt_L \ltw \ssep$.  Also, the following principles are strongly Weihrauch equivalent:
$\ssep$, $\ssep_L$, $\sca$, and $\sca_L$.
\end{thm}

\begin{proof}
The equivalences between the statements and the versions restricted to trees with leaf sets follow from
Lemma \ref{B5} and Theorem \ref{B2}.
The equivalence of $\ssep$ and $\sca$ is included in Theorem 3.11 of Kihara, Marcone, and Pauly \cite{kmp},
while $\ptt \ltws \ssep$ follows from their Corollary 3.7, Theorem 3.11, and Proposition 6.4 \cite{kmp}.
\end{proof}

We close the section with the Weihrauch analog of Theorem \ref{B4}.

\begin{thm}\label{B7}
$\wf \eqws \wf_L$.  Also,
the following four principles are strongly Weihrauch equivalent:
$\wfh$, $\wfh_L$, $\pk$, and $\pk_L$. 
\end{thm}

\begin{proof}
The equivalences $\wf \eqws \wf_L$, $\wfh \eqws \wfh_L$, and $\pk \eqws \pk_L$ all follow
immediately from Lemma \ref{B5} and Theorem \ref{B2}.  It suffices to show that
$\wfh \eqws \pk$.

To see that $\wfh \lews \pk$, let $\langle T_i \rangle_{i \in \nat}$ be a sequence of trees,
the input for $\wfh$.  For sequences $\sigma$ and $\tau$ with $\lh{\sigma}=\lh{\tau}$, let
$\sigma * \tau$ denote the sequence consisting of alternating entries of $\sigma$ and $\tau$.
Thus for $\sigma$ and $\tau$ of length $n+1$,
$\sigma*\tau = \langle \sigma(0) , \tau (0) , \dots , \sigma(n) , \tau (n) \rangle$.
Define the tree $T$ by including the following sequences for each $i \in \nat$:
\begin{list}{$\bullet$}{}
\item $\langle i \rangle \in T$ for each $i \in \nat$, and
\item if $\sigma \in T_i$, $\lh {\sigma} = n$, and $\tau$ is a binary sequence of length $n$,
then $\langle i \rangle \cat (\sigma *\tau )\in T$ and the initial segment of $\langle i \rangle \cat( \sigma *\tau)$
omitting the last element is also in $T$.
\end{list}
The tree $T$ is uniformly computable from the sequence $\langle T_i \rangle_{i \in \nat}$.
If $T_i$ has an infinite path $p$, then for every binary sequence $\tau$, all initial segments
of $\langle i \rangle \cat (p * \tau)$ are in $T$.  In this case, there is a perfect subtree of $T$
above $\langle i\rangle$, so $\langle i\rangle$ is in the perfect kernel of $T$.  If $T_i$ is well-founded,
then the subtree of extensions of $\langle i \rangle$ in $T$ is also well-founded, so no
perfect subtree of $T$ contains $\langle i \rangle$.  Thus, if $K$ is a perfect kernel for $T$,
then $T_i$ is well-founded if and only if $\langle i \rangle \in K$.  Summarizing,
$Z = \{i \mid \langle i \rangle \in K \}$ is the desired output for $\wfh$.

To see that $\pk \lews \wfh$, let $T$ be an input tree for $\pk$.  In the following,
we freely conflate finite sequences with their natural number codes.  For each finite
sequence $\sigma \in T$, define $T_\sigma$ as follows:
\begin{list}{$\bullet$}{}
\item $\langle \sigma \rangle \in T_\sigma$, and
\item if $\langle \sigma , \dots, \langle \tau_0 , \dots , \tau_m \rangle \rangle \in T_\sigma$,
and for each $i \le m$, $\tau_i \cat e_{i,0}$ and $\tau_i \cat e_{i,1}$ are incompatible extensions
of $\tau_i$ in $T$, then
\[
\langle \sigma , \dots  , \langle \tau_0 , \dots , \tau_m \rangle,
\langle \tau_0 \cat e_{0,0}, \tau_0\cat e_{0,1}, \dots, \tau_m \cat e_{m,0}, \tau_m \cat e_{m,1} \rangle\rangle \in T_\sigma .
\]
\end{list}
The sequence $\langle T_\sigma \rangle _{\sigma \in T}$ (which can be viewed as $\langle T_i \rangle_{i \in \nat}$)
is uniformly computable from $T$.  For each $\sigma \in T$, $T_\sigma$ has an infinite path if and only if
$\sigma$ is contained in a perfect subtree of $T$.  Let $Z$ be a solution of $\wfh$ for $\langle T_\sigma \rangle _{\sigma \in T}$.
Then $Z = \{ \sigma \in T \mid T_\sigma {\text{~is well-founded}}\}$, and
$K = \{ \sigma \in T \mid \sigma \notin Z \}$ is the perfect kernel of $T$.
\end{proof}

\subsubsection*{An application}

This section presents a Weihrauch analysis closely related to Theorem 6 of Davis, Hirst, Pardo, and Ransom \cite{davisetal}.
A hypergraph $H = (V, E)$ consists of a set of vertices $V= \{ v_0 , v_1 , \dots \}$ and a set of edges
$E=\{e_1 , e_2 , \dots \}$, where each edge in $E$ is a set of vertices.  For hypergraphs, an edge can be a set of any
cardinality.  If every edge of a hypergraph has cardinality exactly $2$, then $H$ is a graph.
A $k$-coloring of a hypergraph $H=(V,E)$ is a function $f: V \to k$.  A $k$-coloring is called proper
if every edge with at least two vertices contains vertices of different colors.  Let $\hpc (k)$ be the problem that accepts
a hypergraph $H$ as input, outputs $1$ if $H$ has a proper $k$-coloring, and outputs $0$ otherwise.

The second part of the proof of the following theorem applies a leaf management result from the preceding section.

\begin{thm}\label{C1}
For all $k\ge 2$, $\hpc (k) \eqws \wf$.
\end{thm}

\begin{proof}
To see that $\hpc (k) \lews \wf$, let $H = (V,E)$ be a hypergraph input for $\hpc (k)$.
In the following, we freely conflate vertices and finite collections of vertices with their integer codes.
Build a tree
$T$ by including sequences $\sigma = \langle \sigma_0 , \sigma_1 , \dots , \sigma_m \rangle$ satisfying
the following conditions for each $i\le m$.
\begin{list}{$\bullet$}{}
\item If $i = 2j$, then $\sigma _i$ is a set of two vertices in edge $e_j$, or
$\sigma_i$ is a code for $\emptyset$ and $e_j$ does not contain a pair of
vertices in the list $\{ v_0 , \dots , v_m \}$.
\item  If $i = 2j+1$, then $\sigma_i < k$.  We view this as a color for $v_j$.
\item  The partial coloring of $H$ given by the odd entries of $\sigma$ uses
distinct colors on the pairs of vertices listed in the even entries.
\end{list}
The odd entries of any infinite path in $T$ encode a proper $k$-coloring of $H$.  Also,
any proper $k$-coloring of $H$ can be used to define an infinite path through $T$.
(If $H$ has no edges of cardinality less than $2$, an infinite path can be uniformly
computed from any proper coloring, but this is not necessary for the current argument.)
Thus, $\hpc (k)$ is $1$ for $H$ if and only if
$\wf$ is $0$ for $T$.

By Theorem \ref{B7}, $\wf \eqws \wf_L$, so to complete the proof it suffices
to show that $\wf_L \lews \hpc (k)$.  We will prove this for $k=2$ and indicate
how to modify the argument for larger values of $k$.

Let $T$ be an input for $\wf_L$, that is, a tree with a leaf set.
Emulating the construction from the proof of Theorem 6 of
Davis et al.~\cite{davisetal}, define a hypergraph $H$ as follows.
The vertices of $H$ include the five vertices $ \{ a_0, a_1, b_0, b_1 , s\}$
plus two vertices labeled $\sigma_0$ and $\sigma_1$ for each sequence $\sigma \in T$.
The edges of $H$ consist of
\begin{list}{$\bullet$}{}
\item  $(a_0, a_1 )$, $(a_1, s)$, $(b_0 ,b_1 )$, and $(b_1 , s)$,
\item  $(\sigma_0 , \sigma_1 )$ for every nonempty $\sigma \in T$,
\item  $(\sigma_1 , s)$ if $\sigma$ is a leaf of $T$,
\item $E_\sigma = \{ \sigma_1 \}\cup\{ \tau_0  \mid \tau\in T \land \exists n ~ \tau = \sigma\cat n\}$ if $\sigma\in T$ is not a leaf, and
\item $E_0 = \{a_0, b_0\} \cup \{ \sigma_0 \mid \sigma\in T \land \lh {\sigma } = 1\}$.
\end{list}
$H$ is uniformly computable from $T$ and its leaf set.  Note that the leaf set is used in the third and fourth bullets.
$H$ has a proper $2$-coloring if and only if $T$ has
an infinite path.  For details, see the proof of Theorem 6 of \cite{davisetal}.
Thus $\wf_L \lews \hpc (2)$.  To prove the reduction for larger values of $k$,
modify the construction by adding a complete subgraph on $k-2$ vertices to $H$
and connecting each of its vertices to every vertex of $H$ with an edge consisting of two vertices.
\end{proof}

Parallelization yields the Weihrauch analog of part of the reverse mathematical Theorem 6 of Davis et al.~\cite{davisetal}.

\begin{coro}\label{C2}
For $k\ge 2$, $\wfh \eqws \hpch (k)$.
\end{coro}

\begin{proof}
Fix $k$.  By Theorem \ref{C1}, $\wf \eqws \hpc (k)$, so by Proposition 3.6 part (3) of
Brattka, Gherardi, and Pauly \cite{bgp}, $\wfh \eqws \hpch (k)$
\end{proof}

The arguments used in Theorem 6 of \cite{davisetal} can be used to extend Theorem \ref{C1}
and Corollary \ref{C2} to conflict-free colorings.

\subsubsection*{Acknowledgements}

A talk related to this paper was presented at
Dagstuhl
Seminar 18361,
organized by Vasco Brattka, Damir Dzhafarov, Alberto Marcone, and Arno Pauly,
and held September 2-7 of 2018 at Schloss Dagstuhl,
the Leibniz-Zentrum f\"u{}r Informatik \cite{dagstuhl}.
The author's travel to the seminar was supported by
a Board of Trustees travel grant from Appalachian State University.


\begin{bibsection}[Bibliography]
\begin{biblist}

\bib{dagstuhl}{article}{
  author =	{Brattka,Vasco},
  author = {Dzhafarov, Damir},
  author = {Marcone, Alberto},
  author = {Pauly, Arno},
  title =	{{Measuring the Complexity of Computational Content: From Combinatorial Problems to Analysis (Dagstuhl Seminar 18361)}},
  journal =	{Dagstuhl Reports},
  year =	{2018},
  volume =	{8},
  editor =	{Vasco Brattka and Akitoshi Kawamura and Alberto Marcone and Arno Pauly},
  publisher =	{Schloss Dagstuhl--Leibniz-Zentrum fuer Informatik},
  address =	{Dagstuhl, Germany},
  note={To appear},
  }


\bib{bg}{article}{
   author={Brattka, Vasco},
   author={Gherardi, Guido},
   title={Effective choice and boundedness principles in computable
   analysis},
   journal={Bull. Symbolic Logic},
   volume={17},
   date={2011},
   number={1},
   pages={73--117},
   issn={1079-8986},
   review={\MR{2760117}},
   doi={10.2178/bsl/1294186663},
}

\bib{bgp}{article}{
  author={Brattka, Vasco},
   author={Gherardi, Guido},
   author={Pauly, Arno},
   title={Weihrauch Complexity in Computable Analysis},
   year={2017},
   pages={50+xi},
   eprint={arXiv:1707.03202},
}



\bib{davisetal}{article}{
author={Davis, Caleb},
author={Hirst, Jeffry},
author={Pardo, Jake},
author={Ransom, Tim},
title={Reverse mathematics and colorings of hypergraphs},
date={November 29, 2018},
journal={Archive for Mathematical Logic},
DOI={10.1007/s00153-018-0654-z},
}



\bib{fh}{article}{
   author={Friedman, Harvey M.},
   author={Hirst, Jeffry L.},
   title={Weak comparability of well orderings and reverse mathematics},
   journal={Ann. Pure Appl. Logic},
   volume={47},
   date={1990},
   number={1},
   pages={11--29},
   issn={0168-0072},
   review={\MR{1050559}},
   doi={10.1016/0168-0072(90)90014-S},
}

\bib{kmp}{article}{
author={Kihara, Takayuki},
author={Marcone, Alberto},
author={Pauly, Arno},
title={Searching for an analogue of $\text{ATR}_0$ in the {W}eihrauch lattice},
date={2018},
eprint={arXiv:1812.01549}
}




\bib{simpson}{book}{
   author={Simpson, Stephen G.},
   title={Subsystems of second order arithmetic},
   series={Perspectives in Logic},
   edition={2},
   publisher={Cambridge University Press, Cambridge; Association for
   Symbolic Logic, Poughkeepsie, NY},
   date={2009},
   pages={xvi+444},
   isbn={978-0-521-88439-6},
   review={\MR{2517689}},
   doi={10.1017/CBO9780511581007},
}


\end{biblist}
\end{bibsection}
\end{document}